\documentclass[a4paper,11pt]{amsart}
\usepackage{amsfonts}
\usepackage{}
\usepackage{amssymb}
\usepackage{mathrsfs}
\usepackage{amsmath,amssymb,amsthm,latexsym,amscd,mathrsfs}
\usepackage{indentfirst}
 \newcommand{\ROM}[1]{\mathrm{\uppercase\expandafter{\romannumeral#1}}}
  \theoremstyle{definition}
 \newtheorem{thm}{Theorem}[section]
 \newtheorem{lem}{Lemma}[section]
 
 \newtheorem{cor}{Corollary}[section]
  \newtheorem{examp}{Example}[section]
 \newtheorem{rem}{Remark}[section]
 \newtheorem{prop}{Proposition}[section]
 
  \newtheorem{prob}[thm]{Problem}

\newtheorem{ack}{Acknowledgements}   

\title[Isoparametric functions and exotic spheres]{\textbf{Isoparametric functions and exotic spheres}}
\author[J. Q. Ge]{Jianquan Ge}\address{School of Mathematical Sciences, Laboratory of Mathematics and Complex Systems, Beijing Normal
University, Beijing 100875}\email{jqge@bnu.edu.cn}
\author[Z. Z. Tang]{Zizhou Tang}\address{School of Mathematical Sciences, Laboratory of Mathematics and Complex Systems, Beijing Normal University, Beijing 100875}\email{zztang@mx.cei.gov.cn}
\thanks {The project is partially supported by the NSFC (No.11071018 and No.11001016), the SRFDP, and the Program for Changjiang Scholars and Innovative Research Team in University.}
\thanks{The second author is the corresponding author.}
\date{}
\begin{document}
\maketitle

 \begin{center}
 Dedicated to Professor Banghe Li on his 70th birthday.
 \end{center}

\maketitle
\begin{abstract}
The first part of the paper is to improve the fundamental theory of
isoparametric functions on general Riemannian manifolds. Next we
focus our attention on exotic spheres, especially on ``exotic"
$4$-spheres (if exist) and the Gromoll-Meyer sphere. In particular,
as one of main results we prove： there exists no properly
transnormal function on any exotic $4$-sphere if it exists.
Furthermore, by projecting an $S^3$-invariant isoparametric function
on $Sp(2)$, we construct a properly transnormal but not an
isoparametric function on the Gromoll-Meyer sphere with two points
as the focal varieties.

\end{abstract}
%------------------------------------------------------------------------
\section{Introduction}

A hypersurface $M^n$ in a real space form $N^{n+1}(c)$ with constant sectional curvature $c$ is said to be \emph{isoparametric} if
it has constant principal curvatures. Since the work of Cartan (\cite{Ca38},\cite{Ca39}) and M{\"u}nzner (\cite{Mu80}),
the subject of isoparametric hypersurfaces especially in the spherical case is rather fascinating to geometers. They found that each isoparametric hypersurface in a sphere determines a so-called isoparametric function (or Cartan polynomial) that satisfies certain equations, the so-called Cartan-M{\"u}nzner equations, and conversely the set of level hypersurfaces of an isoparametric function consists of a family of parallel isoparametric hypersurfaces (cf. \cite{CR85}). From then on, besides geometrical and topological viewpoints, people also put extensive attention on algebraical approach to this area (see \cite{Th00} for an excellent survey and see \cite{CCJ07}, \cite{Imm08}, \cite{GX10} for recent progresses and applications).

To study such submanifold geometry in a Riemannian manifold more
general than space forms, such as Terng and Thorbergsson \cite{TT95}
did in symmetric spaces where they generalized the notion of
\emph{isoparametric} to \emph{equifocal}, requires much more
calculations and analysis, sometimes even lacking any effective way
to compute some local invariants like covariant derivatives, shape
operators, mean curvatures, \emph{etc}. Following the work of Wang
\cite{Wa87}, we study the subject of isoparametric functions on
general Riemannian manifolds, especially on exotic spheres. Recall
that an $n$-dimensional smooth manifold $\Sigma^n$ is called an
\emph{exotic n-sphere} if it is homeomorphic but not diffeomorphic
to $S^n$. It is Milnor \cite{Mil56} who firstly discovered an exotic
7-sphere which is an $S^3$-bundle over $S^4$. In fact, the
$S^3$-bundles over $S^4$ are in one-to-one correspondence with
elements of $\pi_3SO(4)\cong \mathbb{Z}\oplus\mathbb{Z}$. Let $M^7$
be the total space of an $S^3$-bundle corresponding to
$(m,n)\in\mathbb{Z}\oplus\mathbb{Z}$. Milnor found that when
$m+n=1$, $M^7$ is homeomorphic to $S^7$, and if further $(m-n)^2$ is
not congruent to $1$ $(mod~7)$, $M^7$ is not diffeomorphic to $S^7$
and thus it's an exotic 7-sphere. Later, Kervaire and Milnor
\cite{KM63} computed the group of homotopy spheres in each dimension
greater than four which implies that there exist exotic spheres in
infinitely many dimensions and in each case there are at most
finitely many exotic spheres. In particular, ignoring orientation
there exist 14 exotic 7-spheres, 10 of which can be exhibited as
$S^3$-bundles over $S^4$, the so-called \emph{Milnor spheres}. In
1966, Brieskorn \cite{Br66} was able to realize many odd-dimensional
exotic spheres as subsets of standard spheres. For example, one of
the so-called \emph{Brieskorn varieties} is defined by equation:
$z_0^d+z_1^2+\cdots +z_n^2=0$ for $(z_0,\cdots,z_n)\in
S^{2n+1}\subset \mathbb{C}^{n+1}$ and certain integers $d,n$ (for
more details, see Example \ref{Brieskorn-examp} in Section
\ref{section-isop}). In dimension four, however, the question of
whether an exotic 4-sphere possibly exists, and if so, how many
there are still remains open today. This is essentially the
\emph{smooth Poincar\'{e} conjecture} in dimension four (cf.
\cite{JW08}). As for Riemannian geometry on exotic spheres,
especially for sectional curvature, Gromoll and Meyer \cite{GM74}
produced the first example of an exotic 7-sphere with a metric of
non-negative sectional curvature, which is now called the
\emph{Gromoll-Meyer sphere}. By constructing invariant metrics of
non-negative sectional curvature on cohomogeneity one manifolds with
codimension two singular obits, and applying this to the associated
principle bundles of the Milnor spheres, Grove and Ziller
\cite{GZ00} proved that all 10 Milnor spheres admit metrics with
non-negative sectional curvature. In 2008, Petersen and Wilhelm
\cite{PW08} showed that there is a metric on the Gromoll-Meyer
sphere with positive sectional curvature.

 In this paper, first we improve the fundamental theory of isoparametric functions on general Riemannian manifolds. In Section \ref{section-isop}, based on an
original result of Wang \cite{Wa87} which asserts that singular level sets of a transnormal function are submanifolds and each
regular level set is a tube over them, we describe further topological and geometrical properties of level sets of a
transnormal or an isoparametric function. For example, in Theorem \ref{proper equiv-def} we show: Each component of the singular sets
has codimension not less than $2$ if and only if the singular sets are exactly the focal set of every regular level set; Moreover in
this case, each level set is connected; If in addition the function is isoparametric on a closed manifold, then at least one level hypersurface is minimal. Furthermore, we observe three simple ways to construct examples of isoparametric functions, \emph{i.e.},
(1) For a Riemannian manifold $(N, ds^2)$ with an isoparametric function $f$, take a special conformal deformation
$\widetilde{ds^2}=e^{2u(f)}ds^2$. Then $f$ is also isoparametric on $(N, \widetilde{ds^2})$; (2) For a cohomogeneity one manifold
$(N,G)$ with a $G$-invariant metric, taking composition of some smooth functions on $N/G$ with the projection $\pi:N\rightarrow
N/G$, we get isoparametric functions on $N$; (3) For a Riemannian submersion $\pi: E\rightarrow B$ with minimal fibers, if $f$ is an
isoparametric function on $B$, then so is $F:=f\circ \pi$ on $E$. Applying the second and the third method, we get examples of isoparametric functions on Brieskorn varieties and on isoparametric hypersurfaces of spheres.

Next we focus our attention on such Riemannian geometry on exotic
spheres, especially on ``exotic" 4-spheres (if exist) and the
Gromoll-Meyer sphere. In particular, as one of our main results we
prove in Section \ref{section-4sphere}:
\begin{thm}\label{homotopy 4-sphere}
Suppose $\Sigma^4$ is a homotopy $4$-sphere and it admits a properly transnormal function under some metric. Then $\Sigma^4$ is diffeomorphic to $S^4$.
\end{thm}
Note that a homotopy n-sphere is a smooth manifold with the same homotopy type as $S^n$. Freedman \cite{Fr82} showed that any homotopy 4-sphere is homeomorphic to $S^4$. As a result of this, the above theorem says equivalently that there exists no properly transnormal function (see Section \ref{section-isop} for the definition) on any exotic 4-sphere if it exists. On the other hand, in Section \ref{section-7sphere}, we are able to construct many examples of isoparametric functions on the Milnor spheres. Furthermore, as another one of our main results (see Theorem \ref{isop-on-GM sphere}), by projecting an $S^3$-invariant isoparametric function on $Sp(2)$, we construct a properly transnormal but not an isoparametric function on the Gromoll-Meyer sphere with two points as the focal varieties, which differs from the case occurring on $S^7$. As a consequence, we pose a question that whether there is an isoparametric function on the Gromoll-Meyer sphere or any exotic n-sphere ($n>4$) with two points as the focal varieties. More generally, we pose Problem \ref{problem} that whether there always exist an isoparametric function on an exotic n-sphere ($n>4$) with the focal varieties being those occurring on $S^n$.

\section{Isoparametric functions on Riemannian manifolds}\label{section-isop}
We start with fundamental definitions. A non-constant smooth function $f: N\rightarrow \mathbb{R}$ defined on a Riemannian manifold $N$ is called \emph{transnormal} if there is a smooth function $b:\mathbb{R}\rightarrow\mathbb{R}$ such that
\begin{equation}\label{iso1}
|\nabla f|^2=b(f),
\end{equation}
where $\nabla f$ is the gradient of $f$. If moreover there is a continuous function $a:\mathbb{R}\rightarrow\mathbb{R}$ such that
\begin{equation}\label{iso2}
\triangle f=a(f),
\end{equation}
where $\triangle f$ is the Laplacian of $f$, then $f$ is called \emph{isoparametric} (cf. \cite{Wa87}).
Equation (\ref{iso1}) means that the regular hypersurfaces $M_t:=f^{-1}(t)$ (where $t$ is any regular value of $f$) are parallel and (\ref{iso2}) says that these hypersurfaces have constant mean curvatures. In fact, the first assertion comes from the observation that in this case the integral curves to the vector field $\nabla f/|\nabla f|$ (where $|\nabla f|\neq0$) are geodesics in $N$, and the second is due to the following relation between the shape operator $A$ of $M_t$ and the Hessian $H_f$ of $f$: \begin{equation}\label{shape-hessian}
\langle AX, Y\rangle=-\frac{H_f(X,Y)}{|\nabla f|},
\end{equation}
 where $X$ and $Y$ are tangent vectors to $M_t$. We call these parallel hypersurfaces $M_t$ with constant mean curvatures a family of \emph{isoparametric hypersurfaces}. Note that though $N$ could be non-orientable, the normal bundle $\nu(M_t)$ of each regular level hypersurface $M_t$ of a transnormal function $f$ must be orientable since $\nu:=\nabla f/|\nabla f|$ is a global unit normal vector field of $M_t$ in $N$. Therefore, the normal exponential map $exp: \nu(M_t)\cong M_{t}\times \mathbb{R}\rightarrow N$ which is the restriction of the exponential map $exp$ of $N$ to $\nu(M_t)$ can be written as:
  \begin{equation}\label{normal exp}
  exp(p,s)=exp_p(s\nu)\quad for \quad (p,s)\in M_{t}\times \mathbb{R}.
  \end{equation}

Given a transnormal function $f: N\rightarrow \mathbb{R}$, we denote by $C_1(f)$ the set where $f$ attains its global maximum value or global minimum value, by $C_2(f)$ the union of singular level sets of $f$, \emph{i.e.}, $C_2(f)=\{p\in N| \nabla f(p)=0\}$, and for any regular value $t$ of $f$, by $C^t_3(f)$ the focal set of the level hypersurface $M_t:=f^{-1}(t)$, \emph{i.e.}, the set of singular values of the normal exponential map. By virtue of Lemma $3$ and Lemma $1$ of \cite{Wa87}, it's easily seen that $C_1(f)=C_2(f)$ which was called the \emph{focal varieties} of $f$ in \cite{Wa87}, and $C^{t_1}_3(f)=C^{t_2}_3(f)$ for any two regular level hypersurfaces which will be thus denoted simply by $C_3(f)$. Furthermore, one can find that $C_3(f)\subset C_1(f)=C_2(f)$ in general. For transnormal functions on general Riemannian manifolds,
Wang proved the following fundamental result:
\begin{thm}\label{tube}(\cite{Wa87})
Let $N$ be a connected complete Riemannian manifold and $f$ a transnormal function on $N$. Then
\begin{itemize}
\item[a)] The focal varieties of $f$ are smooth submanifolds (may be disconnected) of $N$;
\item[b)] Each regular level set of $f$ is a tube over either of the focal varieties (the dimensions of the fibers may differ on different connected components).
\end{itemize}
\end{thm}
Let $[\alpha,\beta]\subset\mathbb{R}$ denote the image of $f$ where $\alpha$ or $\beta$ may be the infinity. When $\alpha$ (resp. $\beta$) doesn't equal to the infinity, we set $M_{-}=M_{\alpha}=f^{-1}(\alpha)$ (resp. $M_{+}=M_{\beta}=f^{-1}(\beta)$) and $\emptyset$ otherwise. Thus $C_1(f)=C_2(f)=M_{-}\cup M_{+}$, and Theorem \ref{tube} states that $M_{\pm}$ are smooth submanifolds of $N$ (though may be disconnected and may have different dimension at each component) and $M_t$ is a tube over either of $M_{\pm}$ for any $t\in(\alpha,\beta)$. Furthermore, we observe the following
\begin{thm}\label{proper equiv-def}
Each component of $M_{\pm}$ has codimension not less than $2$ if and
only if $C_3(f)=C_1(f)=C_2(f)$. Moreover in this case, each level
set $M_t$ is connected. If in addition $N$ is closed and $f$ is
isoparametric, then at least one isoparametric hypersurface is
minimal in $N$.
\end{thm}
\begin{proof}
First we show the equivalence assertion. If each component of
$M_{\pm}$ has codimension not less than $2$, then any regular level
hypersurface $M_{t_0}$ is a tube over $M_{\pm}$ with at least $1$
dimensional fiber spheres, and therefore the normal exponential map
$exp: \nu(M)\cong M_{t_0}\times \mathbb{R}\rightarrow N$ defined in
(\ref{normal exp}) is critical whenever the $\mathbb{R}$-component
takes values of the (resp. minus) radius of the tube over $M_{+}$
(resp. $M_{-}$), which implies $C_3(f)=C_1(f)=C_2(f)$. The converse
is also true since the normal exponential map at the distance of the
radius of either tube is a submersion from $M_{t_0}$ to $M_{\pm}$.

When each component of $M_{\pm}$ has codimension not less than $2$,
$N-M_{-}\cup M_{+}\cong M_{t_0}\times(\alpha,\beta)$ is connected
and thus each regular level hypersurface is connected, which implies
that $M_{\pm}$ are connected.

At last, we come to show that, if $f$ is an isoparametric function
on a closed Riemannian manifold $N$ with codimension of $M_{\pm}$
(denoted by $Codim.M_{\pm}$) not less than 2, at least one
isoparametric hypersurface is minimal in $N$. Clearly $M_{\pm}$ are
not empty sets now. By formula (\ref{shape-hessian}), it's easily
calculated that the mean curvature $h(t)$ on the isoparametric
hypersurface $M_t=f^{-1}(t)$ is
$$h(t)=\frac{1}{2\sqrt{b(t)}}(b'(t)-2a(t)).$$ It follows from Lemma
6 of \cite{Wa87} that the eigenvalues of the Hessian $H_f$ of $f$ on
$M_{-}$ (resp. $M_{+}$) are zeros and $\frac{1}{2}b'(\alpha)$ (resp.
$\frac{1}{2}b'(\beta)$) with multiplicities being the dimension and
codimension of $M_{-}$ (resp. $M_{+}$) respectively, which implies
$$a(\alpha)=\triangle
f|_{M_{-}}=Trace(H_f)|_{M_{-}}=\frac{1}{2}b'(\alpha)Codim.M_{-},$$
$$a(\beta)=\triangle
f|_{M_{+}}=Trace(H_f)|_{M_{+}}=\frac{1}{2}b'(\beta)Codim.M_{+}.$$ Recalling
that $b'(\alpha)>0$ and $b'(\beta)<0$ proved in \cite{Wa87}, when
$Codim.M_{\pm}\geq2$, we have
$$b'(\alpha)-2a(\alpha)=b'(\alpha)(1-Codim.M_{-})<0,$$ $$b'(\beta)-2a(\beta)=b'(\beta)(1-Codim.M_{+})>0,$$
 which
confirms the existence of some $t_0\in(\alpha,\beta)$ such that
$$h(t_0)=\frac{1}{2\sqrt{b(t_0)}}(b'(t_0)-2a(t_0))=0.$$
\end{proof}
\begin{rem}\label{exceptional-isop}
Define $f: S^3\rightarrow \mathbb{R}$ by $f(x_0,x_1,x_2,x_3)=x_0^2$. Then direct calculations show that $f$ is isoparametric with $M_{-}=S^2$ and $M_{+}=\{(\pm1,0,0,0)\}$, while $C_3(f)=M_{+}\varsubsetneq C_1(f)=C_2(f)=M_{-}\cup M_{+}$. Clearly there're no minimal regular level hypersurfaces in this example. The isoparametric hypersurfaces $M_t\subset S^3$ here are disconnected, while their projections are just the spheres when we consider $f$ as an isoparametric function on $\mathbb{R}P^3$ though still $C_3(f)\varsubsetneq C_1(f)=C_2(f)$. But we have
\end{rem}
\begin{prop}\label{simplyconn-prop}
Suppose that $N$ is simply connected and $f$ is a transnormal function on $N$ with one regular hypersurface $M_{t_0}=f^{-1}(t_0)$ connected. Then $C_3(f)=C_1(f)=C_2(f)$.
\end{prop}
\begin{proof}
From the proof of Corollary 11.4 and Theorem 11.3 in \cite{MS74}, we
know that there exists no non-orientable hypersurface closed as a
subset embedded in a simply connected manifold. So if
$C_3(f)\varsubsetneq C_1(f)=C_2(f)$, or equivalently, some component
of $M_{\pm}$ is a hypersurface (which is clearly closed as a
subset), then this hypersurface must be orientable and its normal
sphere bundle is the trivial $S^0$ bundle which is obviously
disconnected. On the other hand, recall that $M_{t_0}$ is
diffeomorphic to the normal sphere bundle of either of $M_{\pm}$ and
thus is disconnected, which contradicts the assumption that
$M_{t_0}$ is connected.
\end{proof}
Remark \ref{exceptional-isop} shows that exceptional isoparametric
functions with focal varieties not really focal could exist.
Therefore, from now on, we call a transnormal (isoparametric)
function $f$ \emph{proper} if the focal varieties have codimension
not less than $2$, or equivalently, $C_3(f)=C_1(f)=C_2(f)$. Then
Proposition \ref{simplyconn-prop} asserts that a transnormal
function on a simply connected manifold with one regular level
hypersurface connected is proper. In conclusion, all level sets of a
properly transnormal function $f$ defined on a (connected)
Riemannian manifold $N$ are connected. Furthermore, when $f$ is a
properly isoparametric function,
$\{M_t=f^{-1}(t)|t\in(\alpha,\beta)\}$ consists of a family of
parallel hypersurfaces with constant mean curvature in $N$ (and at
least one minimal if $N$ is closed) and will be called a family of
\emph{properly isoparametric hypersurfaces}, and $M_{\pm}$ are
called the \emph{focal submanifolds}. Thus the level sets of $f$
give a ``singular" foliation of $N$ as $N=\bigcup_{t\in
[\alpha,\beta]}M_t$. Similar to Theorem C asserted in \cite{Wa87},
there is essentially a correspondence between (properly)
isoparametric functions on $N$ and such ``singular" foliations of
$N$ as the union of a family of (properly) isoparametric
hypersurfaces and focal submanifolds (\emph{i.e.}, families of
parallel constant mean curvature hypersurfaces that together with at
most two common focal submanifolds would fill up the whole manifold
$N$, in other words, transnormal systems of codimension 1 on $N$
with all regular foils having constant mean curvature). As is well
known that when $N$ is a real space form, this definition of
isoparametric hypersurface coincides with the usual one that a
hypersurface is isoparametric if it has constant principal
curvatures, since a hypersurface with all its neighboring parallel
translations having constant mean curvature has constant principal
curvatures and vice versa. A startling corollary of Theorem
\ref{tube} in \cite{Wa87} is (see \cite{Miy10} for a detailed
proof):
\begin{thm}(\cite{Wa87})\label{trans-isop-sphere}
Regular level hypersurfaces of a (properly) transnormal function on
$S^n$ or $\mathbb{R}^n$ are isoparametric.
\end{thm}
We will show specifically in Section 4 a properly transnormal
function on the Gromoll-Meyer sphere which is not isoparametric.

Now given an (properly) isoparametric function $f$ on an $n$
dimensional Riemannian manifold $(N^n, ds^2)$ satisfying equation
(\ref{iso1}) and (\ref{iso2}), we consider a \emph{special conformal
deformation} of the given Riemannian metric $ds^2$ of $N$ by
$\widetilde{ds^2}=e^{2u(f)}ds^2$, where $u:\mathbb{R}\rightarrow
\mathbb{R}$ is a smooth function. A straightforward verification
shows that $f$ is still isoparametric on $(N,\widetilde{ds^2})$ with
\begin{eqnarray}
|\widetilde{\nabla}f|^2&=&e^{-2u(f)}b(f), \nonumber\\
\widetilde{\triangle}f~~~~&=&e^{-2u(f)}((n-2)u'(f)b(f)+a(f)),\nonumber
\end{eqnarray}
where $\widetilde{\nabla}f$, $\widetilde{\triangle}f$ denote the gradient and Laplacian of $f$ with respect to $\widetilde{ds^2}$. As a consequence, we have
 \begin{prop}\label{special conformal}
 There always exist infinite Riemannian metrics admitting (properly) isoparametric functions on a fixed manifold once there exists one.
 \end{prop}
\begin{rem}\label{special conformal in Sn}
 Note that when $(N^n,ds^2)$ is a real space form, under a special conformal deformation with respect to an isoparametric function $f$, $(N^n,\widetilde{ds^2}=e^{2u(f)}ds^2)$ may not remain to be a real space form, but the isoparametric hypersufaces $M_t=f^{-1}(t)$ still have constant principal curvatures which could be derived from formula (\ref{shape-hessian}) and the equality $$\widetilde{H_f}(X,Y)=H_f(X,Y)+u'(f)b(f)\langle X, Y\rangle-2u'(f)X(f)~Y(f),$$ where $X,Y$ are tangent vectors of $N^n$, $\langle,\rangle$ is the metric $ds^2$ and $H_f$, $\widetilde{H_f}$ are the Hessians of $f$ under $ds^2$, $\widetilde{ds^2}$ respectively.
\end{rem}
The following two propositions give us effective ways to construct
examples of (properly) isoparametric functions (explicit functions
or implicitly represented by a family of isoparametric
hypersurfaces) on certain Riemannian manifolds. Recall that a
connected manifold $N$ is said to have cohomogeneity one if it
supports a smooth action by a compact Lie group $G$, such that the
orbit space $N/G$ is one-dimensional. We observe the following
\begin{prop}\label{cohom-one-isop}
In a cohomogeneity one manifold $(N^n,G)$ with a $G$-invariant metric, the principal orbits consist of a family of isoparametric hypersurfaces with constant principal curvatures. Moreover, when the non-principal orbits have codimension not less than $2$, they are proper, \emph{i.e.}, the corresponding isoparametric functions with all orbits as their level sets are proper.
\end{prop}
\begin{proof}
The desiring conclusions follow directly from the $G$-invariance of the
Riemannian metric (see also \cite{TT72}, \cite{GZ02}). In fact the
corresponding isoparametric functions could be defined by taking
composition of some smooth functions on $N/G$ with the projection
$\pi:N\rightarrow N/G$.
\end{proof}
\begin{prop}\label{isop-lift}
Let $\pi: E\rightarrow B$ be a Riemannian submersion with minimal
fibers. Then given any (properly) isoparametric function $f$ on $B$,
$F:=f\circ \pi$ is an (properly) isoparametric function on $E$.
\end{prop}
\begin{proof}
Let $\nabla$, $\triangle$ (resp. $\widetilde{\nabla}$, $\widetilde{\triangle}$) denote the gradient and laplace operator on $B$ (resp. $E$), and $H$ the mean curvature vector fields of the fibers. The assertion follows from
$$|\widetilde{\nabla}F|^2=|\nabla f|^2,\quad \quad \quad \widetilde{\triangle}F=\triangle f\circ\pi-\langle H, \widetilde{\nabla} F\rangle.$$
\end{proof}
To conclude this section, now we'll apply the above two propositions to construct examples of isoparametric functions on some interesting manifolds, where we would also state them in terms of isoparametric hypersurfaces and focal submanifolds instead of isoparametric functions for the sake of geometrical viewpoint.
\begin{examp}\label{Brieskorn-examp}
Isoparametric hypersurfaces in Brieskorn varieties.
\end{examp}
One of the $(2n-1)$-dimensional Brieskorn varieties $V^{2n-1}\subset
S^{2n+1}\subset \mathbb{C}^{n+1}$ is defined by the equation
$z_0^d+z_1^2+\cdots +z_n^2=0$ for $(z_0,\cdots,z_n)\in  S^{2n+1}$.
For certain odd $n$ and $d$ (for instance, $n+1\neq 2^k$ for any $k$ and $d\equiv \pm3~mod~8$ ),
the Brieskorn variety is diffeomorphic to an exotic Kervaire
sphere. It's well known that the Brieskorn variety carries a
cohomogeneity one action by $SO(2)SO(n)$ defined by
$$(e^{i\theta},A)(z_0,\cdots,z_n)=(e^{2i\theta}z_0,e^{id\theta}(z_1,\cdots,z_n)A),$$
whose non-principal orbits have codimensions $2$ and $n-1$ (cf.
\cite{HH67}). Thus by Proposition \ref{cohom-one-isop}, the
principal orbits consist of a family of isoparametric hypersurfaces
in the Brieskorn variety. Note that the Brieskorn variety is
diffeomorphic to $S^5$ when $n=3$ and $d$ is odd, hence there are
infinitely many rather different Riemannian metrics on $S^5$ (of
non-negative curvature, cf. \cite{GZ00}) that admit properly
isoparametric functions with corresponding isoparametric
hypersurfaces having constant principal curvatures (compare with
Remark \ref{special conformal in Sn}).

Next we show a very interesting construction.
\begin{examp}\label{isop in isop}
Isoparametric hypersurfaces in isoparametric hypersurfaces of
spheres.
\end{examp}
Recall that Cartan (\cite{Ca38}, \cite{Ca39}) classified all isoparametric hypersurfaces in spheres with $3$ distinct principal curvatures. Such isoparametric hypersurface must be a tube of constant radius over a standard Veronese embedding of a
projective plane $\mathbb{F}P^2$ into $S^{3m+1}$, where $\mathbb{F}$ is the division algebra $\mathbb{R}$, $\mathbb{C}$, $\mathbb{H}$
(quaternions), $\mathbb{O}$ (Cayley numbers) for $m = 1, 2, 4, 8$, respectively. Let $f: S^{3m+1}\rightarrow \mathbb{R}$ be the restriction to $S^{3m+1}$ of the corresponding Cartan polynomial. Then $M^{3m}:=f^{-1}(0)$ is the isoparametric hypersurface with $3$ distinct constant principal curvatures $\cot\frac{\pi}{6}, \cot\frac{\pi}{2}, \cot\frac{5\pi}{6}$ of multiplicities $m$ and $M_{-}\cong M_{+}\cong \mathbb{F}P^2$. Furthermore, direct calculations show that the focal map $\varphi:M\rightarrow M_{-}$ defined by parallel translation of $M$ at distance $\frac{\pi}{2}$ in direction $\nu=\nabla f/|\nabla f|$ is horizontally homothetic, more precisely, we have $$|\varphi_{*}(X)|=\sqrt{3}|X|\quad \emph{for any horizontal vector}\quad X\in (Ker(\varphi_{*}))^{\perp}\subset TM.$$ Moreover, the fibers are totally geodesic (see Corollary 4.12 in \cite{CR85}). Finally we could apply Proposition \ref{isop-lift} to get isoparametric hypersurfaces $\widetilde{M}_t$ in $M^{3m}$ by taking inverse of isoparametric hypersurfaces $\overline{M}_t:=\bar{f}^{-1}(t)\subset\mathbb{F}P^2$ for any isoparametric function $\bar{f}$ on $\mathbb{F}P^2$ under the projection $\varphi$, \emph{i.e.}, $\widetilde{M}_t=\varphi^{-1}(\overline{M}_t)\subset M^{3m}$. Now let's take a look at the case of $m=1$ and the case of $m=8$ for example:
\begin{itemize}
\item[(1)] $m=1$. Define $\bar{f}: \mathbb{R}P^2\rightarrow \mathbb{R}$ by $\bar{f}([x_0,x_1,x_2]):=x_0^2$. Then a simple calculation shows that $\bar{f}$ is an improperly isoparametric function with $\overline{M}_{-}=\mathbb{R}P^1\cong S^1$, $\overline{M}_{+}=\{[1,0,0]\}$ and $\overline{M}_t=\bar{f}^{-1}(t)\cong S^1$ for $t\in(0,1)$. Thus $\bar{f}\circ \varphi$ is an improperly isoparametric function on $SO(3)/\mathbb{Z}_2\oplus \mathbb{Z}_2\cong M^3\subset S^4$ with the focal varieties $\widetilde{M}_{-}=\varphi^{-1}(\overline{M}_{-})\cong K^2$, $\widetilde{M}_{+}=\varphi^{-1}(\overline{M}_{+})\cong S^1$ and the isoparametric hypersurface $\widetilde{M}_t=\varphi^{-1}(\overline{M}_t)\cong T^2$ for $t\in(0,1)$, where $K^2$ is the Klein bottle and $T^2$ is the torus.\\
\item[(2)] $m=8$. Recall that $\mathbb{O}P^2\cong F_4/Spin(9)$ and there are two cohomogeneity one actions on it by $Spin(9)\subset F_4$ and $(Sp(3)\times Sp(1))/\mathbb{Z}_2\subset F_4$ respectively (cf. \cite{Tang98}). Thus from Proposition \ref{cohom-one-isop} we get two properly isoparametric functions on $\mathbb{O}P^2$ with the focal varieties being $(\overline{M}_{-}\cong Spin(9)/Spin(8)\cong S^8, \overline{M}_{+}\cong Spin(9)/Spin(9)=\{pt\})$ and $(\overline{M}_{-}\cong S^{11}, \overline{M}_{+}\cong \mathbb{H}P^2)$ respectively, and the isoparametric hypersurfaces for the former case are $\overline{M}_t\cong Spin(9)/Spin(7)\cong S^{15}$. Therefore we can get two properly isoparametric functions on the isoparametric hypersurface $M^{24}\subset S^{25}$. The corresponding properly isoparametric hypersurfaces in the former case are $\widetilde{M}_t=\varphi^{-1}(S^{15})\cong S^{15}\times S^8\subset M^{24}$ with one focal submanifold $\widetilde{M}_{+}=\varphi^{-1}(\{pt\})\cong S^8$ and another topologically an $S^8$-bundle over $S^8$. Similar cases occur on $\mathbb{C}P^2$ and $\mathbb{H}P^2$ if one consider cohomogeneity one actions on them.
\end{itemize}
     \begin{rem}
     Notice that the Hopf fibration $\pi: S^{3m-1}\rightarrow \mathbb{F}P^2$ ($m=1,2,4$) is a Riemannian submersion with totally geodesic fibres and thus isoparametric functions on $\mathbb{F}P^2$ can also be lifted to isoparametric functions on $S^{3m-1}$. Conversely, we can get isoparametric functions on $\mathbb{F}P^2$ by projecting those $S^{m-1}$ invariant isoparametric functions on $S^{3m-1}$ for $m=1,2,4$. For instance, the inhomogeneous examples of Ozeki and Takeuchi \cite{OT75} are invariant under the canonical $S^3$-action and hence also invariant under the canonical $S^1$-action, which give examples for $m=4$ and $m=2$ respectively. See more examples and studies about isoparametric hypersurfaces in projective spaces in \cite{Tak75}, \cite{Pa89} and \cite{Wa82}.
     \end{rem}
     \begin{rem}\label{iterative lift}
     In fact, for any sphere bundle $\pi:E\rightarrow B$ over a Riemannian manifold $B$, one can modify the metric of $E$ such that $\pi$ could become a Riemannian submersion with totally geodesic fibres. For instance, one can firstly use the Kaluza-Klein procedure to define a metric $g_{KK}$ on $P$, the associate principal bundle $\tilde{\pi}:P\rightarrow B$, which is constructed by declaring the horizontal and vertical subspaces orthogonal, and giving the vertical space a biinvariant metric on the structure group $G$ through the connection and the horizontal space the metric of $B$ via $\tilde{\pi}_*$ (cf. \cite{Bo88}, \cite{Je73}, \cite{LY74}). Then given a $G$-invariant metric $g_F$ on the fibre $F$ of $\pi$, the direct product $g:=g_{KK}\times g_F$ is a $G$-invariant metric on $P\times F$. Finally, identifying $E$ with $(P\times F)/G$, by horizontal projection, one induces from $g$ a Riemannian metric on $E$ in desire (see \cite{Po75} or Proposition 3.1 in \cite{Na79}). So one can always lift isoparametric functions on $B$ to get isoparametric functions on $E$, which can be (iteratively) applied to the focal map between an arbitrary isoparametric hypersurface and its focal submanifold though the metrics may altered. Since an isoparametric function give a ``singular" foliation of the manifold, such iterations of lifting above give an ``iterated" ``singular" foliation structure of the manifold.
     \end{rem}
\section{Non-existence on ``exotic" 4-spheres}\label{section-4sphere}
 By a homotopy n-sphere we mean a smooth manifold with the same homotopy type as $S^n$. Freedman \cite{Fr82} showed a remarkable result that any homotopy 4-sphere is homeomorphic to $S^4$; however the question of whether such a manifold is necessarily diffeomorphic to $S^4$ (\emph{i.e.}, whether there is an exotic $S^4$? This is essentially the smooth Poincar\'{e} conjecture in dimension four.) still remains mysterious to mathematicians (cf. \cite{JW08}). In this section we show a small step towards this conjecture, \emph{i.e.}, Theorem \ref{homotopy 4-sphere} stated in the introduction, which asserts that a homotopy 4-sphere $\Sigma^4$ must be diffeomorphic to $S^4$ if it admits a properly transnormal function.

 Our first task is to determine all possibilities of the two focal varieties that a properly transnormal function would give on the standard $S^4$, where we arrange $M_{\pm}$ such that the codimension of $M_{+}$ is not less than that of $M_{-}$.
\begin{prop}\label{focal in S4}
Let $f: S^4\rightarrow \mathbb{R}$ be a properly transnormal
function. Then the focal varieties $(M_{+}, M_{-})$ and the regular
level hypersurfaces $M_t$ must be one of the following
\begin{itemize}
\item[(1)] $(\{pt\}, \{pt\})$,\quad $S^3$;
\item[(2)] $(S^1, S^2)$, \quad$S^1\times S^2$;
\item[(3)] $(\mathbb{R}P^2,\mathbb{R}P^2)$,\quad $SO(3)/\mathbb{Z}_2\oplus\mathbb{Z}_2$.
\end{itemize}
\end{prop}
\begin{proof}
 Theorem \ref{trans-isop-sphere} tells us that the regular hypersurfaces $M_t=f^{-1}(t)$ are isoparametric hypersurfaces in $S^4$ and $(M_{-}, M_{+})$ are just the focal submanifolds. Thus the conclusion follows from the classification of isoparametric hypersurfaces in $S^4$ with $1$, $2$ or $3$ distinct principal curvatures.
\end{proof}
Now let's recall a topological theorem of M\"{u}nzner
\cite{Mu80}. (See \cite{CR85}, p289.)
\begin{thm}(\cite{Mu80})\label{munzner}
Let $M$ be a compact connected hypersurface in $S^{n+1}$ such that:
\begin{itemize}
\item[(a)] $S^{n+1}$ is divided into two manifolds $(B_1,M)$ and $(B_{-1},M)$ with boundary $M$.
\item[(b)] For $k=\pm1$, $B_k$ has the structure of a differentiable ball bundle over a compact manifold $M_k$ of dimension $n-m_k$.
\end{itemize}
Let the ring of coefficients $R=\mathbb{Z}$ if $M_{\pm1}$ are both orientable and $\mathbb{Z}_2$ otherwise. Let $\mu=m_1+m_{-1}$. Then $\alpha:=2n/\mu$ is an integer, and for $k=\pm1$,
$$H^q(M_k)=\left\{
        \begin{array}{l}
          R ~\quad for~ q\equiv0~~ \quad \quad (mod ~\mu),~ 0\leqslant q<n;\\
          R ~\quad for~ q\equiv m_{-k}\quad (mod ~\mu),~ 0\leqslant q<n;\\
          0 ~\quad otherwise.
        \end{array}
  \right.$$
  Further,
  $$H^q(M)=\left\{
        \begin{array}{l}
          R ~\quad for~ q=0, ~n;\\
          H^q(M_1)\oplus H^q(M_{-1}), ~\quad for~ 1\leqslant q\leqslant n-1.
        \end{array}
  \right.$$
\end{thm}
 Observe that if a homotopy 4-sphere $\Sigma^4$ admits a properly transnormal function $f$, we know from Section 2 that every regular level hypersurface $M_t=f^{-1}(t)$ is connected and divides $\Sigma^4$ into two ball bundles over the focal varieties $M_{\pm1}$ of codimensions $m_{\pm1}+1\geqslant 2$. By a similar argument as the proof of Theorem \ref{munzner}, we have
\begin{prop}
Suppose a homotopy 4-sphere $\Sigma^4$ admits a properly transnormal
function $f$. Let $M_{t_0}$ be a regular level hypersurface of $f$
and $M_1$, $M_{-1}$ be the two focal varieties of codimensions
$m_{1}+1\geqslant m_{-1}+1(\geqslant 2)$ respectively. Then one of
the following happens:
\begin{itemize}
\item[(1)] $(m_1,m_{-1})=(3,3)$, \quad$(M_1,M_{-1})=(\{pt\},\{pt\})$,\quad $M_{t_0}\cong S^3$;
\item[(2)] $(m_1,m_{-1})=(2,1)$, \quad$H^0(M_1)=H^1(M_1)=\mathbb{Z}$, \quad$H^0(M_{-1})=H^2(M_{-1})=\mathbb{Z}$ and $H^1(M_{-1})=0$, \quad $H^0(M_{t_0})=H^1(M_{t_0})=H^2(M_{t_0})=H^3(M_{t_0})=\mathbb{Z}$;
\item[(3)] $(m_1,m_{-1})=(1,1)$, \quad$H^0(M_1)=H^1(M_1)=H^2(M_1)=\mathbb{Z}_2$, \quad$H^0(M_{-1})=H^1(M_{-1})=H^2(M_{-1})=\mathbb{Z}_2$, \quad$H^0(M_{t_0})=H^3(M_{t_0})=\mathbb{Z}_2$ and $H^1(M_{t_0})=H^2(M_{t_0})=\mathbb{Z}_2\oplus \mathbb{Z}_2$.
\end{itemize}
\end{prop}
Because 2-dimensional manifolds are determined by their cohomology structures and $M_{t_0}$ is diffeomorphic to the normal sphere bundle over either of $M_{\pm1}$, we have
\begin{cor}\label{3cases}
With the same notations as above, one of the following happens:
\begin{itemize}
\item[(1)] $(M_1,M_{-1})=(\{pt\},\{pt\})$,\quad $M_{t_0}\cong S^3$;
\item[(2)] $(M_1,M_{-1})\cong(S^1,S^2)$,\quad $M_{t_0}\cong S^1\times S^2$;
\item[(3)] $(M_1,M_{-1})\cong(\mathbb{R}P^2,\mathbb{R}P^2)$,\quad $M_{t_0}\cong SO(3)/\mathbb{Z}_2\oplus\mathbb{Z}_2$.
\end{itemize}
\end{cor}
\begin{proof}
It suffices to verify each case of $M_{t_0}$. The fist case is
obvious. The second case comes from the fact that $S^1$-bundles over
$S^2$ are the lens spaces which are determined by their cohomology
structures (cf. \cite{JW54}). The third case is due to Proposition
\ref{focal in S4} and a theorem of Massey \cite{Ma74} which states
that for any embedding of $\mathbb{R}P^2$ in $\Sigma^4$, the normal
bundle is unique and independent of the embedding\footnote{Although
Massey only proved it for $S^4$, it turns out that it's also true
for any homotopy 4-sphere due to his topological proof.}. In fact, this normal
bundle is isomorphic to the normal bundle of the Veronese embedding of $\mathbb{R}P^2$ in $S^4$ due to a result of Levine \cite{Le63} that: 2-plane bundles over $\mathbb{R}P^2$ are completely determined by the first Stiefel-Whitney class and the twisted Euler class which are equal to $1\in H^1(\mathbb{R}P^2,\mathbb{Z}_2)\cong\mathbb{Z}_2$ (\cite{Ma74}) and $2\in H^2(\mathbb{R}P^2,\mathcal{Z})\cong \mathbb{Z}$ (\cite{Ma69}) respectively for both normal bundles in $\Sigma^4$ and $S^4$, where we use $\mathcal{Z}$ as the symbol for twisted integer coefficients. We remark that
all the equalities are diffeomorphisms as a well known fact says
that every topological manifold of dimension one, two or three
admits a unique smooth structure (see \cite{FQ90}).
\end{proof}
To prove Theorem \ref{homotopy 4-sphere} we need only verify it respectively for these three cases in Corollary \ref{3cases}. Recall that for each case, we can write
$\Sigma^4$ as $\Sigma^4=B_1\cup_{M_{t_0}}B_{-1}$, where $B_{\pm1}$
are ball bundles over $M_{\pm1}$ respectively with the same boundary
$M_{t_0}$. Therefore, in the first case we can look $\Sigma^4$ as a
twisted 4-sphere, where a twisted n-sphere $M^n$, by definition, is a manifold
constructed by gluing two discs $D_1^n$, $D_2^n$ along their
boundaries using a diffeomorfism
$h:S^{n-1}\rightarrow S^{n-1}$ and is usually written as
$M^n=D_1^n\cup_hD_2^n$. It's well known that a twisted sphere must
be homeomorphic to the standard sphere (cf. \cite{Mil56}). Thus a celebrated theorem of Cerf \cite{Ce68} which states that every twisted 4-sphere is
diffeomorphic to $S^4$ helps us verify the first case.

As for the second case, \emph{i.e.}, when $(M_1,M_{-1})\cong(S^1,S^2)$ and $M_{t_0}\cong S^1\times S^2$, we have $B_1\cong S^1\times D^3$, $B_{-1}\cong D^2\times S^2$ and $\Sigma^4\cong (S^1\times D^3)\cup_h(D^2\times S^2)$, where $h: S^1\times S^2\rightarrow S^1\times S^2$ is the diffeomorphism induced from the diffeomorphisms between the common boundary of $B_{\pm1}$ and the boundaries of $S^1\times D^3$ and $D^2\times S^2$. On the other hand, we have the standard decomposition of $S^4$ by certain isoparametric hypersurface with $2$ principal curvatures, \emph{i.e.}, $S^4= (S^1\times D^3)\cup_{id}(D^2\times S^2)$, where $id: S^1\times S^2\rightarrow S^1\times S^2$ is the identity map. Fortunately again Theorem 17.1 and its corollary in \cite{Gl62} help us verify this case that $\Sigma^4\cong S^4$. In fact, firstly one can extend $h$ to a diffeomorphism, say $H: S^1\times D^3\rightarrow S^1\times D^3$, such that $H|_{\partial(S^1\times D^3)}=h$. Then he can construct a diffeomorphism $\Phi: S^4\rightarrow\Sigma^4$ as follows: $\Phi|_{S^1\times D^3}=H: S^1\times D^3\rightarrow S^1\times D^3$, $\Phi|_{D^2\times S^2}=id: D^2\times S^2\rightarrow D^2\times S^2$.
\begin{rem}
Gluck \cite{Gl62} also proved that $\mathcal{H}(S^1\times S^2)=\mathbb{Z}_2\oplus\mathbb{Z}_2\oplus\mathbb{Z}_2$, where $\mathcal{H}(S^1\times S^2)$ denotes the group of diffeomorphism types of $S^1\times S^2$. The corresponding generators are
$$\begin{array}{l}
          S^1\times S^2\rightarrow S^1\times S^2,\\
          (z,\quad w)\mapsto (\bar{z},\quad w)
        \end{array}
        \begin{array}{l}
          S^1\times S^2\rightarrow S^1\times S^2,\\
          (z,\quad w)\mapsto (z,~~ -w)
        \end{array}
        \begin{array}{l}
          S^1\times S^2\rightarrow S^1\times S^2,\\
          (z,\quad w)\mapsto (z,~~ \phi(z)w)
        \end{array}
                $$
where $\phi(z)w=(z\cdot(w_1,w_2),w_3)$ is a rotation of $w=(w_1,w_2,w_3)\in S^2$ around the $w_3$-axis through the angle of $z\in S^1$. It's easily seen that each generator can be extended to a diffeomorphism of $S^1\times D^3$. Thus every diffeomorphism of $S^1\times S^2$ extends to a diffeomorphism of $S^1\times D^3$.
\end{rem}
It remains to prove the third case, \emph{i.e.}, the case when $(M_1,M_{-1})\cong(\mathbb{R}P^2,\mathbb{R}P^2)$, $M_{t_0}\cong SO(3)/\mathbb{Z}_2\oplus\mathbb{Z}_2$. Note that $B_{\pm1}$ are just the (unique) normal disk bundle of the embeddings of $\mathbb{R}P^2$ in $\Sigma^4$ and thus $\Sigma^4\cong B_1\cup_hB_{-1}$, where $h: \partial B_{-1}\cong SO(3)/\mathbb{Z}_2\oplus\mathbb{Z}_2\rightarrow \partial B_1\cong SO(3)/\mathbb{Z}_2\oplus\mathbb{Z}_2$ is the induced diffeomorphism. As argued in the proof of Corollary \ref{3cases}, $B_{\pm1}$ are diffeomorphic to the normal disk bundle of the Veronese embedding of $\mathbb{R}P^2$ in $S^4$. Finally, Price \cite{Pr77} showed that such manifold as $\Sigma^4$ glued by two copies of the normal disk bundle through a diffeomorphism of $SO(3)/\mathbb{Z}_2\oplus\mathbb{Z}_2$ is diffeomorphic to $S^4$ or has $\mathbb{Z}_2$ as its fundamental group, which thus helps us verify the third case and finish the proof of Theorem \ref{homotopy 4-sphere}.
\begin{rem}
In fact, $\Sigma^4$ in the second and the third case is just the
$4$-manifold constructed through some so-called \emph{Gluck surgery}
and \emph{Price surgery} respectively, where in brief the Gluck
surgery (resp. Price surgery) in $S^4$ is a method to construct a
new 4-manifold from a given embedding of $S^2$ (resp.
$\mathbb{R}P^2$) in $S^4$ by regluing $N(S^2)$ (resp.
$N(\mathbb{R}P^2)$) back to $S^4-N(S^2)$ (resp.
$S^4-N(\mathbb{R}P^2)$) through a self-diffeomorphism of
$\partial(N(S^2))\cong S^1\times S^2$ (resp.
$\partial(N(\mathbb{R}P^2))\cong
SO(3)/\mathbb{Z}_2\oplus\mathbb{Z}_2$), where $N(S^2)$ (resp.
$N(\mathbb{R}P^2)$) is a tubular neighborhood of the embedding of
$S^2$ (resp. $\mathbb{R}P^2$) that is diffeomorphic to the normal
disk bundle (cf. \cite{KSTY99}). The condition of admitting a
properly transnormal function provides $\Sigma^4$ with a further
restriction during its construction through these surgeries, that is
the exterior of the embedding (\emph{i.e.}, $S^4-N(S^2)$ or
respectively $S^4-N(\mathbb{R}P^2)$) is also a tubular neighborhood
of an embedding of $S^2$ (resp. $\mathbb{R}P^2$), and thus yields no
new 4-manifolds other than $S^4$.
\end{rem}
We conclude this section with a worthy note that our proof of
Theorem \ref{homotopy 4-sphere} above is rather dispersed and
depends on many known results, a natural question left is to give a
direct proof of the existence of a diffeomorphism between $\Sigma^4$
and $S^4$ via the given transnormal function on $\Sigma^4$. A more
tempting problem is to find a properly transnormal function on every
homotopy 4-sphere under some metric, which together with Theorem
\ref{homotopy 4-sphere} would give an affirmative answer to the
smooth Poincar\'{e} conjecture in dimension four.
\section{Existence on exotic 7-spheres}\label{section-7sphere}
We have shown in last section that exotic 4-spheres (if exist) admit
no properly transnormal function. Recall that in Section 2 we have
examples of isoparametric functions on exotic Brieskorn spheres as
an application of Proposition \ref{cohom-one-isop}. In this section,
we concentrate on existence problem of isoparametric functions on
exotic 7-spheres, especially on the Gromoll-Meyer sphere.

It is well known that ten of the $14$ exotic 7-spheres (ignoring
orientation) can be exhibited as $S^3$-bundles over $S^4$, the
so-called Milnor spheres. Thus by modifying metrics on the Milnor
spheres such that they are the total spaces of Riemannian submersions with totally
geodesic fibres over the standard $S^4$ if needed (which is always
possible, see Remark \ref{iterative lift}), as Example \ref{isop in isop} in Section 2, applying Proposition \ref{isop-lift} we can get
\begin{prop}
There exist properly isoparametric functions on each Milnor sphere
with the corresponding isoparametric hypersurfaces and focal
submanifolds diffeomorphic to those in $S^7$ gotten from the inverse
of Hopf fibration $\pi:S^7\rightarrow S^4$.
\end{prop}
\begin{rem}
Note that isoparametric hypersurfaces and focal submanifolds in
$S^4$ are completely listed in Proposition \ref{focal in S4}, and
every sphere bundle over $S^4$ is trivial when restricted on them.
Thus the isoparametric hypersurfaces and focal submanifolds in the
Milnor spheres or $S^7$ coming from inverse of these ones are all
diffeomorphic to the product of $S^3$ with the initial ones.
\end{rem}
\begin{rem}
One may be interested in those Riemannian metrics on the Milnor
spheres that admit properly isoparametric functions and make the
fibrations over $S^4$ be Riemannian submersions with totally
geodesic fibres. An intriguing class of such metrics on a Milnor
sphere are those induced from some invariant metrics on the
cohomogeneity one manifold, the total space of associated principal
bundle, and were proved to have non-negative sectional curvature in
\cite{GZ00}.
\end{rem}
Comparing with the cases occurring on $S^7$, now we put our
attention on constructing a properly isoparametric function on the
Milnor spheres with two points as the focal varieties. Firstly we
focus on the Gromoll-Meyer sphere $\Sigma^7$ which is the unique
exotic 7-sphere as a biquotient and thus we expect to construct such
example by projecting an $S^3$-invariant isoparametric function on
$Sp(2)$ onto $\Sigma^7$. It turns out that this function on
$\Sigma^7$ is not isoparametric though still properly transnormal
under the induced metric from some left invariant metric of $Sp(2)$,
and the level hypersurfaces have non-constant mean curvature, which
differs from the case occurring on $S^7$ (see Theorem
\ref{trans-isop-sphere}). But we don't know whether this function
would be isoparametric under a suitable metric, or would such
isoparametric functions exist. Further, we suspect the existence of
such isoparametric functions with two points as the focal varieties
on any exotic sphere. More generally, it's very interesting to pose
the following:
\begin{prob}\label{problem}
Does there always exist a properly isoparametric function on an
exotic sphere $\Sigma^n$ ($n>4$) with the focal varieties being
those occurring on $S^n$?
\end{prob}
Now we start our construction of a transnormal function as follows.

Let $Sp(2)=\{Q=\left(\begin{array}{cc}
          a& b\\
          c& d
        \end{array}\right)\in M(2,\mathbb{H})|~QQ^*=I \}$ be equipped with a left invariant metric $\langle~,~\rangle$ such that at $T_ISp(2)=\mathfrak{sp}(2)=\{\xi=\left(\begin{array}{cc}
          x& y\\
          -\bar{y}& z
        \end{array}\right)\in M(2,\mathbb{H})|~Re(x)=Re(z)=0 \}$,
        \begin{equation}\label{left inv. metric on Sp2}
        \left|\left(\begin{array}{cc}
          x& y\\
          -\bar{y}& z
        \end{array}\right)\right|^2=|x|^2+|y|^2+|z|^2.\end{equation}
Let $\xi_1,\xi_2,\xi_3$ be three left invariant vector fields on $Sp(2)$ such that at $I$, $\xi_i|_I=\left(\begin{array}{cc}
          x_i& y_i\\
          -\bar{y_i}& z_i
        \end{array}\right)$, $i=1,2,3$. Then calculating the Levi-Civita connection from the equality
        $$2\langle\nabla_{\xi_1}\xi_2,~\xi_3\rangle=\xi_1\langle\xi_2,\xi_3\rangle+\xi_2\langle\xi_1,\xi_3\rangle-\xi_3\langle\xi_1,\xi_2\rangle
        -\langle\xi_1, [\xi_2,\xi_3]\rangle+\langle\xi_2, [\xi_3,\xi_1]\rangle+\langle\xi_3, [\xi_1,\xi_2]\rangle,$$
        we establish
\begin{lem}\label{connection in Sp2}
$$\nabla_{\xi_1}\xi_2=\frac{1}{2}[\xi_1,\xi_2]+D(\xi_1,\xi_2),$$ where $$D(\xi_1,\xi_2)|_I=\frac{1}{2}\left(\begin{array}{cc}
          0& y_1z_2+y_2z_1-x_1y_2-x_2y_1\\
          -\overline{y_1z_2+y_2z_1-x_1y_2-x_2y_1}& 0
        \end{array}\right).$$
\end{lem}
For general vector fields $\xi, \eta$, suppose $\gamma(t)$ is the integral curve of $\eta$ in $Sp(2)$ with $\gamma(0)=Q$, then we can make use of the following formula to calculate $\nabla_{\eta}\xi$:
\begin{equation}\label{generalconn in Sp2}
\nabla_{\eta}\xi|_Q=Q\frac{d}{dt}(\gamma(t)^*\xi(\gamma(t)))|_{t=0}+Q(\frac{1}{2}[Q^*\eta,Q^*\xi]+D(Q^*\eta,Q^*\xi)),
\end{equation}
where $Q^*=\bar{Q}^t=Q^{-1}$ for $Q\in Sp(2)$.

Define the function mentioned before on $Sp(2)$ by
\begin{equation}\label{isop on Sp2}
F(Q):=Re(a) \quad\emph{for}\quad Q=\left(\begin{array}{cc}
          a& b\\
          c& d
        \end{array}\right)\in Sp(2).
        \end{equation}
Later we will show that $F$ is $S^3$-invariant and thus can be projected as a function $f$ on the Gromoll-Meyer sphere $\Sigma^7\cong Sp(2)/S^3$. Now we show firstly the following
\begin{thm}\label{thm in Sp2}
The function $F$ defined by (\ref{isop on Sp2}) is a properly
isoparametric function on $Sp(2)$ under the left invariant
metric\footnote{Note that if we equip $Sp(2)$ with a biinvariant
metric, the function $F$ will be no longer isoparametric or
equifocal in the sense of \cite{TT95}, which should be compared with
Theorem 4.3 in \cite{Tang98} where some confusion happened.} defined
by (\ref{left inv. metric on Sp2}). In fact,
$$|\nabla F|^2=1-F^2, \quad \triangle F=-7F.$$
Further, the isoparametric hypersurface $M_t^9:=F^{-1}(t)\subset Sp(2)$ has three distinct principal curvatures (in general)
$$\lambda_1(Q)=\frac{t}{\sqrt{1-t^2}}, \quad \lambda_2(Q)=\frac{t+\sqrt{t^2+4|b|^2}}{2\sqrt{1-t^2}},\quad \lambda_3(Q)=\frac{t-\sqrt{t^2+4|b|^2}}{2\sqrt{1-t^2}},$$
for $Q=\left(\begin{array}{cc}
          a& b\\
          c& d
        \end{array}\right)\in M_t^9$ and each has multiplicity $3$. At last the focal submanifolds are $$M_{\pm1}=\left\{\left(\begin{array}{cc}
          \pm1& 0\\
          0& d
        \end{array}\right)|~d\in S^3\right\}\cong S^3.$$
\end{thm}
\begin{rem}
The isoparametric hypersurface $M_0^9=F^{-1}(0)\cong S^6\times S^3$ has three distinct principal curvatures $0,~ |b|, ~-|b|$ with the same muliplicity $3$ and thus it's an austere (minimal) hypersurface in $(Sp(2),~\langle~,~\rangle)$ in the sense of Harvey and Lawson \cite{HL82}.
\end{rem}
\begin{proof}
By solving the equation $\langle\nabla F, \xi\rangle=\xi (F)$ for any vector $\xi$, we find
        $$\nabla F|_Q=-Q\left(\begin{array}{cc}
          Im(a)& b\\
          -\bar{b}& 0
        \end{array}\right), \quad \emph{for} \quad Q=\left(\begin{array}{cc}
          a& b\\
          c& d
        \end{array}\right)\in Sp(2),$$
where $Im(a)=a-Re(a)$ is the imaginary part of $a\in\mathbb{H}$. Therefore, $|\nabla F|^2=1-F^2$.

For $\xi_i=Q\left(\begin{array}{cc}
          x_i& y_i\\
          -\bar{y_i}& z_i
        \end{array}\right)\in T_QSp(2)$, $i=1,2$, extending them to left invariant vector fields $\xi_1,\xi_2$, we are ready to calculate the Hessian of $F$ by Lemma \ref{connection in Sp2} as:
\begin{eqnarray}\label{hessian in Sp2}
H_F(\xi_1,\xi_2)&:=&\langle\nabla_{\xi_1}(\nabla F),~ \xi_2\rangle=\xi_1\langle\nabla F,~\xi_2\rangle-\langle\nabla F, ~ \nabla_{\xi_1}\xi_2\rangle\nonumber\\&=&\xi_1\xi_2F-\frac{1}{2}[\xi_1,\xi_2]F-D(\xi_1,\xi_2)F\nonumber\\&=&-\langle Q\left(\begin{array}{cc}
          Fx_1& Fy_1+bz_1\\
          -(\overline{Fy_1+bz_1})& Im(\bar{b}y_1)
        \end{array}\right), ~\xi_2\rangle,
\end{eqnarray}
and thus $\triangle F=Tr(H_F)=-7F$. Meanwhile, it's easily seen that $\nu:=\nabla F/|\nabla F|$ is the unit normal vector field of any level hypersurface $M_t^9=F^{-1}(t)$. Then by formula (\ref{shape-hessian}), we derive the expression of the shape operator $A_{\nu}$ of $M_t^9$:
$$
A_{\nu}(\xi)=\frac{Q}{|\nabla F|}\left(\begin{array}{cc}
          tx& ty+bz\\
          -(\overline{ty+bz})& Im(\bar{b}y)
        \end{array}\right),\quad for \quad \xi=Q\left(\begin{array}{cc}
          x& y\\
          -\bar{y}& z
        \end{array}\right)\in T_QM_t^9.
$$
Now suppose $\lambda$ is an eigenvalue of $A_{\nu}$ at $Q\in M_t^9$, then there exists a nonzero vector $\xi=Q\left(\begin{array}{cc}
          x& y\\
          -\bar{y}& z
        \end{array}\right)\in T_QM_t^9$ such that $A_{\nu}(\xi)=\lambda\xi$, \emph{i.e.},
        $$\frac{1}{|\nabla F|}\left(\begin{array}{cc}
          tx& ty+bz\\
          -(\overline{ty+bz})& Im(\bar{b}y)
        \end{array}\right)=\lambda\left(\begin{array}{cc}
          x& y\\
          -\bar{y}& z
        \end{array}\right).$$
Therefore, when $b=0$ and $t=0$, we have $\lambda\equiv0$ with its eigenspace being $T_QM_t^9$; when $b\neq0$, solving the above equation, we have three solutions $\lambda_1,\lambda_2,\lambda_3$ as specified in the theorem with their eigenspaces being
 \begin{eqnarray}
 T_1&=&\left\{\left(\begin{array}{cc}
          x& \frac{Re(ax)}{|b|^2}b\\
          -\frac{Re(ax)}{|b|^2}\bar{b}& 0
        \end{array}\right)|~ Re(x)=0\right\},\nonumber\\
T_2&=&\left\{\left(\begin{array}{cc}
          0& y\\
          -\bar{y}& \frac{\bar{b}y}{\lambda_2|\nabla F|}
        \end{array}\right)| ~Re(\bar{b}y)=0\right\},\nonumber\\
T_3&=&\left\{\left(\begin{array}{cc}
          0& y\\
          -\bar{y}& \frac{\bar{b}y}{\lambda_3|\nabla F|}
        \end{array}\right)| ~Re(\bar{b}y)=0\right\},\nonumber
\end{eqnarray}
each of which has dimension $3$. Finally, $|\nabla F|=0$ if and only if $F=\pm1$, which completes the proof.
\end{proof}

We are now in a position to give a precise definition of the Gromoll-Meyer sphere. Recall that the Gromoll-Meyer sphere $\Sigma^7$ was defined in \cite{GM74} as a quotient $Sp(2)/S^3$, where for $q\in S^3\subset\mathbb{H}$, \begin{equation}\label{GM-action}
\left(\begin{array}{cc}
          a& b\\
          c& d
        \end{array}\right)\thicksim\left(\begin{array}{cc}
          q& 0\\
          0& 1
        \end{array}\right)\left(\begin{array}{cc}
          a& b\\
          c& d
        \end{array}\right)\left(\begin{array}{cc}
          \bar{q}& 0\\
          0& \bar{q}
        \end{array}\right).\end{equation}
Obviously, $Re(a)=Re(qa\bar{q})$ for any $q\in S^3$, which yields that the function $F$ on $Sp(2)$ defined by (\ref{isop on Sp2}) is $S^3$-invariant and thus can be projected as a function, say $f$, on $\Sigma^7\cong Sp(2)/S^3$. On the other hand, it's easily seen that the metric $\langle~,~\rangle$ on $Sp(2)$ defined by (\ref{left inv. metric on Sp2}) is invariant under the $S^3$-action in (\ref{GM-action}) and thus it induces a metric on $\Sigma^7\cong Sp(2)/S^3$ (still denoted by $\langle~,~\rangle$) making the submersion into a Riemannian submersion.
\begin{thm}\label{isop-on-GM sphere}
The function $f$ on the Gromoll-Meyer sphere $\Sigma^7\cong
Sp(2)/S^3$ under the induced metric is a properly transnormal but
not an isoparametric function with two points as the focal
varieties, and the regular level hypersurfaces of $f$ have
non-constant mean curvature.
\end{thm}
\begin{rem}
In fact, from the proof of Theorem 1 in \cite{GM74}, one can see that this function $f$ is just the Morse function defined by Milnor (\cite{Mil56}, page 404) on the Gromoll-Meyer sphere. Coincidently, the integral curves of $\nabla f/|\nabla f|$ is just the geodesics whose explicit formulas were given by Dur\'{a}n \cite{Du01}.
\end{rem}
\begin{proof}
Though the proof is now direct, the computations are complicated due to the intricacy of the $S^3$-action in (\ref{GM-action}). By definition, $F=f\circ \pi$, where $\pi:Sp(2)\rightarrow\Sigma^7\cong Sp(2)/S^3$ is the projection. Then $\nabla f=d\pi(\nabla F)$ and thus $|\nabla f|^2=1-f^2$. So $f$ is transnormal and the focal varieties are $$M_{\pm1}=f^{-1}(\pm1)=\left\{ \left[ \left(\begin{array}{cc}
          \pm1& 0\\
          0& 1
        \end{array}\right)\right]\right\},$$
where $[\cdot]$ denotes the $S^3$-orbit of a matrix in $Sp(2)$ that seen as a point in $\Sigma^7$. Now it suffices to calculate $\triangle f$ which is rather affected by the intricacy of the $S^3$-action.

Let $H$ denote the mean curvature vector field of an $S^3$-orbit in $Sp(2)$ (\emph{i.e.}, the fibre $\pi^{-1}([Q])$ for some $[Q]\in \Sigma^7$). Then it's easily seen that $\triangle F=\triangle f\circ\pi-\langle H, \nabla F\rangle$. Of course, $\Phi:=\langle H, \nabla F\rangle$ is $S^3$-invariant and thus can be projected as a function, say $\phi$, on $\Sigma^7$. Then since $\triangle F=-7F$ by Theorem \ref{thm in Sp2}, $$\triangle f=-7f+\phi.$$ Therefore we need only to calculate $\phi$ for any $[Q]\in \Sigma^7$. In particular, we will calculate $\phi$ for $[Q]\in f^{-1}(0)\subset \Sigma^7$ and find that it's a non-constant function on $f^{-1}(0)$, which says that $\phi$ (and hence $\triangle f$) is not a function of $f$ and completes the proof.

Observe that the tangent space of an $S^3$-orbit $\pi^{-1}([Q_0])$ at $Q=\left(\begin{array}{cc}
          a& b\\
          c& d
        \end{array}\right)\sim Q_0$ is $$\left\{\nu(x):=Q\left(\begin{array}{cc}
          \bar{a}xa-x& \bar{a}xb\\
          \bar{b}xa& \bar{b}xb-x
        \end{array}\right)|~Re(x)=0\right\}.$$
Then $\{\nu_1:=\nu(i), \nu_2:=\nu(j), \nu_3:=\nu(k)\}$ constitutes a global frame of $\pi^{-1}([Q_0])$. Let $g_{\alpha\beta}:=\langle \nu_{\alpha},\nu_{\beta}\rangle$, for $\alpha, \beta=1,2,3$, and let $(g^{\alpha\beta}):=(g_{\alpha\beta})^{-1}$ be the inverse matrix. Although the canonical frame $\{\nu_1, \nu_2, \nu_3\}$ is not left invariant and thus their covariant derivatives should be calculated by formula (\ref{generalconn in Sp2}), since $\langle\nu_{\alpha},\nabla F\rangle=\nu_{\alpha}(F)=0$, we have
\begin{equation}\label{Phi}
\Phi=\sum g^{\alpha\beta}\langle\nabla_{\nu_{\alpha}}\nu_{\beta},\nabla F\rangle
=-\sum g^{\alpha\beta}H_F(\nu_{\alpha},\nu_{\beta}).
\end{equation}
We'll use the $S^3$-invariance of $\Phi$ to simplify the computations, that is to choose some ``good" point $Q$ in each orbit such that the matrix $(g_{\alpha\beta})$ would get simple. We need the following:
\begin{lem}
For any $a,b\in\mathbb{H}$, there exists some $q\in S^3$ such that $$qb\bar{q}=b_0+b_1i, \quad qa\bar{q}=a_0+a_1i+a_2j, \quad for\quad b_0,b_1, a_0,a_1,a_2\in\mathbb{R}.$$
\end{lem}
The proof is easily seen from the fact that the canonical projection $\varphi:S^3\rightarrow SO(3)$, $\varphi(q):\mathbb{R}^3=Im(\mathbb{H})\rightarrow\mathbb{R}^3=Im(\mathbb{H})$ defined by $x\mapsto qx\bar{q}$, is surjective.

We go on proving our theorem. Without loss of generality, we assume
\begin{equation}\label{Qassump}
a=a_1i+a_2j, \quad b=b_0+b_1i,\quad for~ any \quad [Q]=\left[\left(\begin{array}{cc}
          a& b\\
          c& d
        \end{array}\right)\right]\in f^{-1}(0).
        \end{equation}
Then direct calculations show that
$$(g_{\alpha\beta})=\left(\begin{array}{ccc}
          1-|a|^2|b|^2+4a_2^2& -4a_1a_2&0\\
          -4a_1a_2& 1-|a|^2|b|^2+4(a_1^2+b_1^2)&0\\
          0&0&1-|a|^2|b|^2+4(a_1^2+a_2^2+b_1^2)
        \end{array}\right).$$
Thus
$$(g^{\alpha\beta})=\left(\begin{array}{ccc}
       (1-|a|^2|b|^2+4(a_1^2+b_1^2))/E   & 4a_1a_2/E&0\\
          4a_1a_2/E&(1-|a|^2|b|^2+4a_2^2)/E &0\\
          0&0&1/F
        \end{array}\right),$$
        where $E=(1-|a|^2|b|^2+4a_2^2)(1-|a|^2|b|^2+4(a_1^2+b_1^2))-16a_1^2a_2^2$\quad and  $F=1-|a|^2|b|^2+4(a_1^2+a_2^2+b_1^2)$.
On the other hand, by formula (\ref{hessian in Sp2}), we have
$$H_F(\nu_1,\nu_1)=0,\quad H_F(\nu_2,\nu_2)=H_F(\nu_3,\nu_3)=-4a_1b_1b_0, \quad H_F(\nu_1,\nu_2)=2a_2b_1b_0.$$
Finally we derive the formula for $\phi$ on $f^{-1}(0)$ from (\ref{Phi}), that is for any $[Q]\in f^{-1}(0)$ written as in (\ref{Qassump}),
\begin{equation}\label{laplacian phi}
\phi([Q])=\frac{8a_1b_1b_0}{EF}(E-8a_2^2b_1^2).
        \end{equation}
Now it's clear that $\phi$ (and hence $\triangle f$) is not constant on $f^{-1}(0)$, which also implies $f^{-1}(0)$ has non-constant mean curvature by formula (\ref{shape-hessian}). The proof is now complete.
\end{proof}

From (\ref{laplacian phi}) and (\ref{shape-hessian}) we can get an explicit formula for the mean curvature of $f^{-1}(0)\subset \Sigma^7$. Also in the same way, the function $\phi$ and hence $\triangle f$ can be calculated on the whole of $\Sigma^7$ but with more complicated expression.

\bigskip
\begin{ack}
It's our great pleasure to thank Professor
Chiakuei Peng for many useful communications about calculations of covariant derivatives, and also thank Professors Reiko Miyaoka and Gudlaugur Thorbergsson for their interests on our work.
At last, we are glad to thank Professor Karsten Grove for his nice comments.
\end{ack}

\end{document}